\documentclass{amsart}

\usepackage[margin=1in]{geometry}
\usepackage{tikz}

\begin{document}

\title{Phutball draws}
\author{Sucharit Sarkar}
\address{Department of Mathematics\\University of California\\Los Angeles, CA 90095}
\email{sucharit@math.ucla.edu}

\begin{abstract}In this short note, we exhibit a draw in the game of
  Philosopher's Phutball. We construct a position on a $12\times 10$
  Phutball board from where either player has a drawing strategy, and
  then generalize it to an $m\times n$ board with $m-2\geq n\geq 10$.
\end{abstract}

\maketitle

\newtheorem{theorem}{Theorem}
\newtheorem{nontheorem}[theorem]{Non-Theorem}
\newtheorem{corollary}[theorem]{Corollary}
\newtheorem{question}[theorem]{Question}

\renewcommand{\th}{^{\text{th}}}

\newcommand{\Dsw}{\swarrow}
\newcommand{\Ds}{\downarrow}
\newcommand{\Dse}{\searrow}
\newcommand{\Dw}{\leftarrow}
\newcommand{\De}{\to}
\newcommand{\Dnw}{\nwarrow}
\newcommand{\Dn}{\uparrow}
\newcommand{\Dne}{\nearrow}

\tikzstyle{chap}=[circle,fill=black,minimum height=17pt,inner sep=0pt,
outer sep=0pt,style={transform shape=true}]
\tikzstyle{ball}=[circle,fill=black!20!white,draw=black,thick,minimum
height=17pt,inner sep=0pt, outer sep=0pt,style={transform shape=true}]
\tikzstyle{move}=[rectangle,draw=black,inner sep=1pt, outer sep=0pt,anchor=west]

\nocite{*}

Philosophers' Phutball, invented in the towers of Cambridge by Conway
and company, and named after a much-beloved Monty Python sketch of a
similar name, is a two-player game played on a $m\times n$ board. The
official Phutball pitch, as described by
Berlekamp-Conway-Guy~\cite{BCG}, is $19\times 15$ (that is, there are
$19$ rows and $15$ columns); however, Phutball is usually played on a
$19\times 19$ Go board.

The rules of the game are fairly simple. One of the grid points is
occupied by \emph{the ball}, usually a black Go stone. Alfred, the
first player, wins if the ball crosses the topmost row or ends up in
the topmost row at the end of a turn. Betty, going second, wins if the
ball crosses the bottommost row or ends up in the bottommost row at
the end of a turn. On their turn, the players may either place \emph{a
  chap}, usually a white Go stone, or move the ball. The ball is moved
by jumping over a line of chaps in one of the eight possible
directions, and those chaps are immediately removed; and multiple
jumps are allowed, although not required. However, in all our
diagrams, we will represent the ball by a grey stone and the chaps by
black stones (Phutball played with reversed colors is called floodlit
Phutball).

Despite the simplicity of the rules, the game is fairly
complicated. Phutball on an $n\times n$ board is
PSPACE-hard~\cite{phutball-pspace} and to even check if one has a win
in one is NP-complete~\cite{phutball-np}. Several variants of Phutball
have been analyzed in great detail, such as ``directional
Phutball''~\cite{phutball-directional} when the players are
constrained to jump in certain directions, or ``one-dimensional
Phutball''~\cite{phutball-1d} played on an $m\times 1$ board. Even the
one-dimensional game is surprisingly complicated: It has only been
analyzed fully in~\cite{phutball-1d} when the players are forbidden to
place `off-parity' chaps. In the general one-dimensional game,
sometimes the only winning move for a player is to jump all the way
back! This makes the game `loopy' and hard to analyze. It is not even
known if there are one-dimensional configurations from where both
players have a drawing strategy~\cite{phutball-loopy}.  Even less is
known regarding the usual two-dimensional Phutball. Starting from the
middle on a one-dimensional empty $(2m+1)\times 1$ board, the first
player Alfred can at least ensure a draw by the usual
strategy-stealing argument, similar to the proof
of~\cite[Observation~3]{phutball-directional}, but even the
strategy-stealing argument fails in the two-dimensional case. After
filling up his top row with chaps, Alfred runs out of `passes'. So it
is still unknown if starting from the middle on a two-dimensional
empty $(2m+1)\times n$ board the first player Alfred has a drawing
strategy. 

In this short note, we will study draws in the game of Phutball. The
main result, Theorem~\ref{thm:main}, is that there is a position on a
$12\times 10$ board which leads to a draw under optimal play by both
the players. The configuration immediately generalizes to an
$m\times n$ board with $m-2\geq n\geq 10$, see
Corollary~\ref{cor:main}.

We will use chessboard notation to describe the board. The columns are
labeled A, B, C, \dots, left to right, the rows are numbered 1, 2, 3,
\dots, bottom to top, and the grid points are labeled a1, a2, \dots,
b1, b2, \dots, accordingly. The $k\th$ move by Alfred will be referred
to as $\alpha$($k$), and the $k\th$ move by Betty as $\beta$($k$),
with the $k$ in Roman numerals. A chap placement will simply be
referred to by the name of the grid point where the chap was placed,
while a jump will be described by a (non-empty) sequence of arrows,
where the arrows describe the directions of the jump (read left to
right).

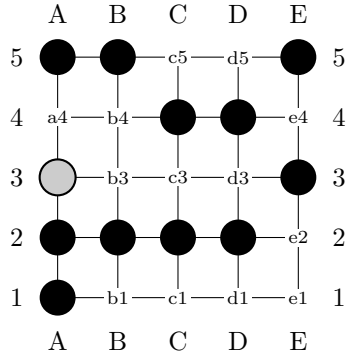
\begin{figure}[ht]
\centering
\begin{tikzpicture}[scale=0.8]

\draw (0,1) grid (4,5);

\foreach \l [count=\i from 0] in {A,B,C,D,E}
{
  \node[below=3ex,inner sep=0pt,outer sep=0pt] at (\i,1) {\l};
  \node[above=3ex,inner sep=0pt,outer sep=0pt] at (\i,5) {\l};
}
\foreach \j in {1,...,5}
{
  \node[left=3ex,inner sep=0pt,outer sep=0pt] at (0,\j) {\j};
  \node[right=3ex,inner sep=0pt,outer sep=0pt] at (4,\j) {\j};
}
\foreach \l [count=\i from 0] in {a,b,c,d,e}
{
  \foreach \j in {1,...,5}
  {
    \node[fill=white,fill opacity=1, text opacity=1,inner sep=1pt,
    outer sep=0pt] (\l\j) at (\i,\j) {\scriptsize \l\j};
  }
}

\node[ball] at (a3) {};

\node[chap] at (b2) {};
\node[chap] at (c2) {};
\node[chap] at (d2) {};
\node[chap] at (e3) {};
\node[chap] at (d4) {};
\node[chap] at (e5) {};
\node[chap] at (a2) {};
\node[chap] at (a1) {};
\node[chap] at (c4) {};
\node[chap] at (a5) {};
\node[chap] at (b5) {};
\end{tikzpicture}
\caption{An illustrative Phutball position.}\label{fig:illustrative}
\end{figure}
Allow us to illustrate the rules, the notations, and a few subtleties
in the following example. Assume we have the board position from
Figure~\ref{fig:illustrative} on a $5\times 5$ board.  At this point,
these are the legal moves: a chap placement on any of the $13$ empty
points, or any of the following jumps, $\Ds$, $\Dse$, $\Dse\Dn$,
$\Dse\Dn\Dse$, $\Dse\Dn\Dn$, $\Dse\Dn\Dne$. The jumps $\Dse\Dn\Dn$ and
$\Dse\Dn\Dne$ win the game for Alfred, while the jumps $\Ds$, $\Dse$,
and $\Dse\Dn\Dse$ win it for Betty. Note that during the jumps
$\Dse\Dn$, $\Dse\Dn\Dn$, and $\Dse\Dn\Dne$, the ball uses the bottom
row, but since it does not end up in the bottom row at the end of the
turn, it is not a win for Betty. Also note that the jumps $\Dse\Dne$,
$\Dse\Dn\Dn\Dw$, and $\Dse\Dn\Dn\Dse$ are not allowed since they hit
the sidelines, but the jump $\Dse\Dn\Dne$ is allowed; the corners are
considered part of the goallines, not the sidelines, and jumping
diagonally through the corners is allowed; jumping through the top
corners is a win for Alfred while jumping through the bottom corners
is a win for Betty. The jump $\Dse\Dn\Dsw$ is also not allowed, since
the chap at b2 is removed immediately after the first jump and cannot
be reused during the third jump.

Let us now discuss an elementary Phutball strategy before we
proceed. If a player has a winning jump, then it is called \emph{a
  shot}. There are only two ways to defend against a shot. One may
perform a jump of their own, jumping away from the danger; we will
call such a jump \emph{a jot}, short for \emph{j}umping \emph{o}ut of
\emph{t}rouble. Alternatively, one may place a chap along the route of
the winning jump, so that the winning jump no longer exists; such a
chap placement is called \emph{a tackle}. If after every jump
(respectively, a chap placement), it is still a shot, perhaps by some
other route, then the original shot is called an \emph{unjottable}
(respectively, \emph{untackleable}) shot. We will annotate shots,
unjottable shots, and untackleable shots by $!$, $\ast!$, and $!!$,
respectively. If a player has a shot that is both unjottable and
untackleable, then he or she has a win in one, and we will indicate
this by $\#$. Consider the situation from Figure~\ref{fig:shot-tackle}
on the $5\times 5$ board.
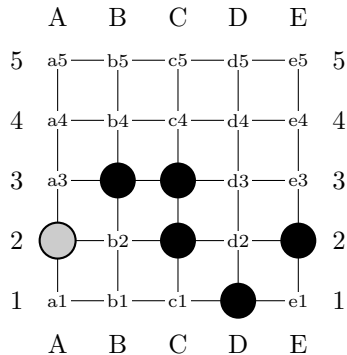
\begin{figure}[ht]
\centering
\begin{tikzpicture}[scale=0.8]

\draw (0,1) grid (4,5);

\foreach \l [count=\i from 0] in {A,B,C,D,E}
{
  \node[below=3ex,inner sep=0pt,outer sep=0pt] at (\i,1) {\l};
  \node[above=3ex,inner sep=0pt,outer sep=0pt] at (\i,5) {\l};
}
\foreach \j in {1,...,5}
{
  \node[left=3ex,inner sep=0pt,outer sep=0pt] at (0,\j) {\j};
  \node[right=3ex,inner sep=0pt,outer sep=0pt] at (4,\j) {\j};
}
\foreach \l [count=\i from 0] in {a,b,c,d,e}
{
  \foreach \j in {1,...,5}
  {
    \node[fill=white,fill opacity=1, text opacity=1,inner sep=1pt,
    outer sep=0pt] (\l\j) at (\i,\j) {\scriptsize \l\j};
  }
}

\node[ball] at (a2) {};

\node[chap] at (b3) {};
\node[chap] at (c3) {};
\node[chap] at (c2) {};
\node[chap] at (d1) {};
\node[chap] at (e2) {};
\end{tikzpicture}
\caption{Shots, tackles, and jots.}\label{fig:shot-tackle}
\end{figure}
It is Alfred's turn, and he has to defend against Betty's shot
$\Dne\Ds$. He can jot $\Dne\Ds\De\Dn$, or he can tackle by placing a
chap at c4. In this case, the tackle wins while the jot loses. By
tackling at c4, Alfred gets the untackleable unjottable shot $\Dne$. On
the other hand, if Alfred had jotted off to e3, then Betty can simply
place a chap on e2 to get a untackleable unjottable shot of her own.

Now we are ready for our main result.

\begin{theorem}\label{thm:main}
  There are drawn configurations in Phutball played on a $12\times 10$
  board.
\end{theorem}

\begin{proof}
  Consider the configuration from Figure~\ref{fig:main-draw}, with
  Alfred to play. (Notice that the configuration of the chaps has a
  symmetry, which will of relevance very soon.)
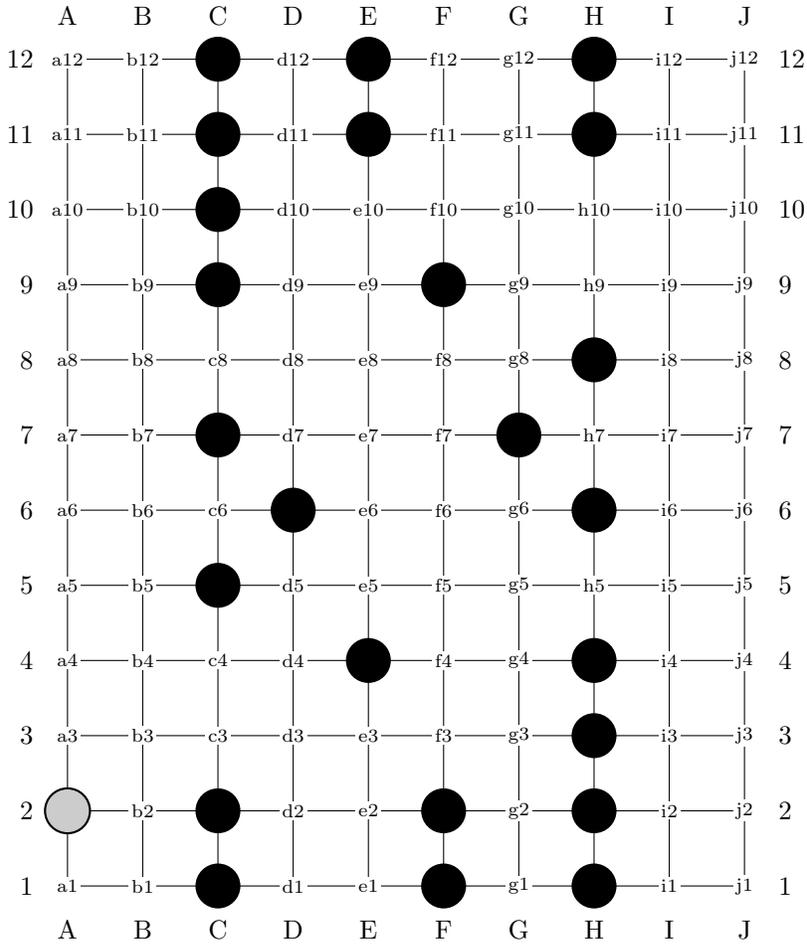
\begin{figure}[ht]
\centering
\begin{tikzpicture}{scale=0.8}

\draw (0,1) grid (9,12);

\foreach \l [count=\i from 0] in {A,B,C,D,E,F,G,H,I,J}
{
  \node[below=3ex,inner sep=0pt,outer sep=0pt] at (\i,1) {\l};
  \node[above=3ex,inner sep=0pt,outer sep=0pt] at (\i,12) {\l};
}
\foreach \j in {1,...,12}
{
  \node[left=3ex,inner sep=0pt,outer sep=0pt] at (0,\j) {\j};
  \node[right=3ex,inner sep=0pt,outer sep=0pt] at (9,\j) {\j};
}
\foreach \l [count=\i from 0] in {a,b,c,d,e,f,g,h,i,j}
{
  \foreach \j in {1,...,12}
  {
    \node[fill=white,fill opacity=1, text opacity=1,inner sep=1pt,
    outer sep=0pt] (\l\j) at (\i,\j) {\scriptsize \l\j};
  }
}

\node[ball] at (a2) {};

\node[chap] at (c1) {};
\node[chap] at (c2) {};
\node[chap] at (c5) {};
\node[chap] at (c7) {};
\node[chap] at (c9) {};
\node[chap] at (c10) {};
\node[chap] at (c11) {};
\node[chap] at (c12) {};

\node[chap] at (d6) {};
\node[chap] at (e4) {};
\node[chap] at (e11) {};
\node[chap] at (e12) {};

\node[chap] at (f1) {};
\node[chap] at (f2) {};
\node[chap] at (f9) {};
\node[chap] at (g7) {};

\node[chap] at (h1) {};
\node[chap] at (h2) {};
\node[chap] at (h3) {};
\node[chap] at (h4) {};
\node[chap] at (h6) {};
\node[chap] at (h8) {};
\node[chap] at (h11) {};
\node[chap] at (h12) {};

\end{tikzpicture}
\caption{A drawn position in Phutball.}\label{fig:main-draw}
\end{figure}

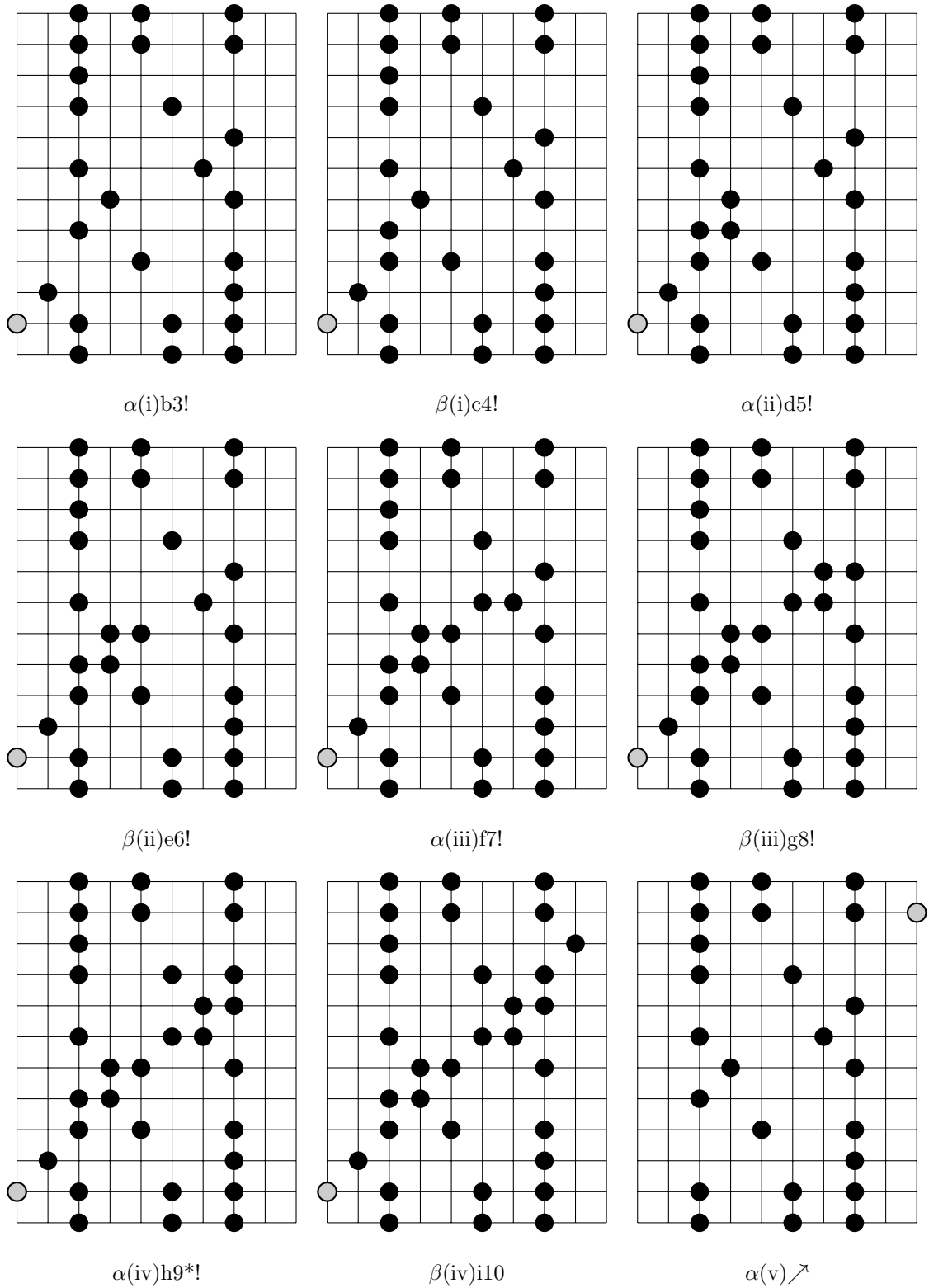
\begin{figure}[ht]
  \centering
  \begin{tikzpicture}[scale=0.5]
    \foreach \b [count=\t from 0] in {$\alpha$(i)b3!,$\beta$(i)c4!,$\alpha$(ii)d5!,$\beta$(ii)e6!,$\alpha$(iii)f7!,$\beta$(iii)g8!,$\alpha$(iv)h9*!,$\beta$(iv)i10,$\alpha$(v)$\Dne$}
    {
      \tikzset{xshift={mod(\t,3)*10cm}, yshift=-floor(\t/3)*14cm}
        \draw (0,1) grid (9,12);
        \node[below=1ex,inner sep=0pt, outer sep=0pt] at (4.5,0) {\b};
        
        \foreach \l [count=\i from 0] in {a,b,c,d,e,f,g,h,i,j}
        {
          \foreach \j in {1,...,12}
          {
            \node[inner sep=0pt,outer sep=0pt] (\t-\l\j) at (\i,\j) {};
          }
        }
        
        \node[chap] at (\t-c1) {};
        \node[chap] at (\t-c2) {};
        \node[chap] at (\t-c5) {};
        \node[chap] at (\t-c7) {};
        \node[chap] at (\t-c9) {};
        \node[chap] at (\t-c10) {};
        \node[chap] at (\t-c11) {};
        \node[chap] at (\t-c12) {};
        
        \node[chap] at (\t-d6) {};
        \node[chap] at (\t-e4) {};
        \node[chap] at (\t-e11) {};
        \node[chap] at (\t-e12) {};
        
        \node[chap] at (\t-f1) {};
        \node[chap] at (\t-f2) {};
        \node[chap] at (\t-f9) {};
        \node[chap] at (\t-g7) {};
        
        \node[chap] at (\t-h1) {};
        \node[chap] at (\t-h2) {};
        \node[chap] at (\t-h3) {};
        \node[chap] at (\t-h4) {};
        \node[chap] at (\t-h6) {};
        \node[chap] at (\t-h8) {};
        \node[chap] at (\t-h11) {};
        \node[chap] at (\t-h12) {};
    }

    \foreach \t in {0,...,7}
    {
      \node[ball] at (\t-a2) {};
      \foreach \i in {0,...,\t}
      {
        \tikzset{xshift={mod(\t,3)*10cm}, yshift=-floor(\t/3)*14cm}
        \node[chap] at (1+\i,3+\i) {};
      }
    }
    \node[ball] at (8-j11) {};

  \end{tikzpicture}
  \caption{Half of the forced sequence of moves in the
    draw.}\label{fig:draw-sequence}
\end{figure}

For her next move, Betty is threatening to place a chap at a1 to get
an untackleable shot. Alfred must jot away from that threat, so he
needs to place a chap next to the ball to create a jump of his
own. Placing a chap at a1, b1, or b2 is a shot for Betty, while
placing a chap at a3 is not much better, since the sequence
\(
\text{$\alpha$(i)a3 $\beta$(i)a1!! $\alpha$(ii)$\Dn$ $\beta$(ii)a3\#}
\)
loses for Alfred. So he must place a chap at b3.

However, this creates a shot for Alfred, so Betty has to defend
against it. Jotting off $\Dne\Dn\De$ to e6 is not helpful since Alfred
can place a chap back at d6 to get an untackleable unjottable shot. So
Betty has to tackle by placing a chap at c4.

The tables turn again. Now Betty has a shot and Alfred has to defend
against it. His only available tackle is at d5, and as the following
game tree shows, if he does not tackle, he loses.
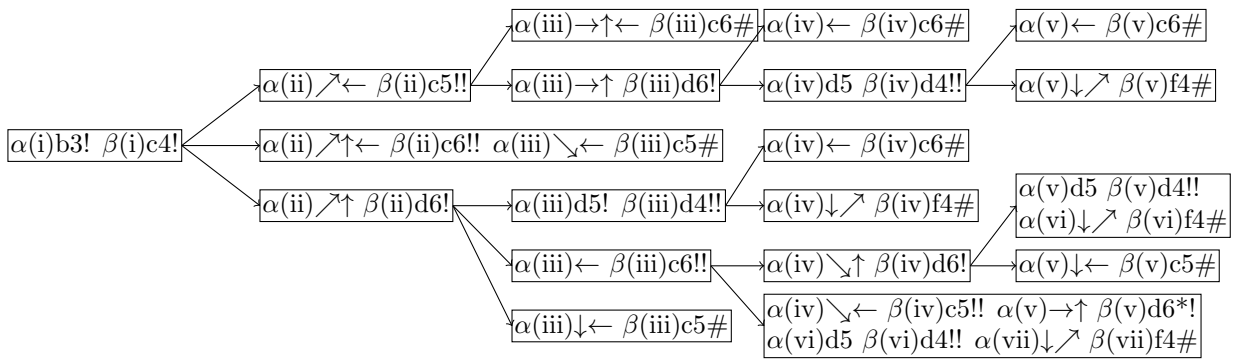
\[
\begin{tikzpicture}[xscale=3.35,yscale=0.8]
\coordinate (S) at (0,1);

\coordinate (b5) at (1,2);
\coordinate (b7) at (1,1);
\coordinate (d7) at (1,0);

\coordinate (b5b7) at (2,3);
\coordinate (b5d7) at (2,2);
\coordinate (b5d7b7) at (3,3);
\coordinate (b5d7T) at (3,2);
\coordinate (b5d7Tb7) at (4,3);
\coordinate (b5d7Tf5) at (4,2);

\node[move] (nS) at (S) {$\alpha$(i)b3! $\beta$(i)c4!};

\node[move] (nb5) at (b5) {$\alpha$(ii)$\Dne\Dw$ $\beta$(ii)c5!!};
\node[move] (nb7) at (b7) {$\alpha$(ii)$\Dne\Dn\Dw$ $\beta$(ii)c6!! $\alpha$(iii)$\Dse\Dw$ $\beta$(iii)c5\#};
\node[move] (nd7) at (d7) {$\alpha$(ii)$\Dne\Dn$ $\beta$(ii)d6!};

\node[move] (nb5b7) at (b5b7) {$\alpha$(iii)$\De\Dn\Dw$ $\beta$(iii)c6\#};
\node[move] (nb5d7) at (b5d7) {$\alpha$(iii)$\De\Dn$ $\beta$(iii)d6!};
\node[move] (nb5d7b7) at (b5d7b7) {$\alpha$(iv)$\Dw$ $\beta$(iv)c6\#};
\node[move] (nb5d7T) at (b5d7T) {$\alpha$(iv)d5 $\beta$(iv)d4!!};
\node[move] (nb5d7Tb7) at (b5d7Tb7) {$\alpha$(v)$\Dw$ $\beta$(v)c6\#};
\node[move] (nb5d7Tf5) at (b5d7Tf5) {$\alpha$(v)$\Ds\Dne$ $\beta$(v)f4\#};

\coordinate (d7b7) at (2,-1);
\coordinate (d7b5) at (2,-2);
\coordinate (d7T) at (2,0);
\coordinate (d7Tb7) at (3,1);
\coordinate (d7Tf5) at (3,0);
\coordinate (d7b7d7) at (3,-1);
\coordinate (d7b7b5) at (3,-2);
\coordinate (d7b7d7T) at (4,0);
\coordinate (d7b7d7b5) at (4,-1);

\node[move] (nd7b7) at (d7b7) {$\alpha$(iii)$\Dw$ $\beta$(iii)c6!!};
\node[move] (nd7b5) at (d7b5) {$\alpha$(iii)$\Ds\Dw$ $\beta$(iii)c5\#};
\node[move] (nd7T) at (d7T) {$\alpha$(iii)d5! $\beta$(iii)d4!!};
\node[move] (nd7Tb7) at (d7Tb7) {$\alpha$(iv)$\Dw$ $\beta$(iv)c6\#};
\node[move] (nd7Tf5) at (d7Tf5) {$\alpha$(iv)$\Ds\Dne$ $\beta$(iv)f4\#};
\node[move] (nd7b7d7) at (d7b7d7) {$\alpha$(iv)$\Dse\Dn$ $\beta$(iv)d6!};
\node[move,align=left] (nd7b7b5) at (d7b7b5) {$\alpha$(iv)$\Dse\Dw$ $\beta$(iv)c5!! $\alpha$(v)$\De\Dn$ $\beta$(v)d6*!\\$\alpha$(vi)d5 $\beta$(vi)d4!! $\alpha$(vii)$\Ds\Dne$ $\beta$(vii)f4\#};
\node[move,align=left] (nd7b7d7T) at (d7b7d7T) {$\alpha$(v)d5 $\beta$(v)d4!!\\$\alpha$(vi)$\Ds\Dne$ $\beta$(vi)f4\#};
\node[move] (nd7b7d7b5) at (d7b7d7b5) {$\alpha$(v)$\Ds\Dw$ $\beta$(v)c5\#};

\draw[->] (nS.east) -- (nb5.west);
\draw[->] (nS.east) -- (nb7.west);
\draw[->] (nS.east) -- (nd7.west);

\draw[->] (nb5.east) -- (nb5b7.west);
\draw[->] (nb5.east) -- (nb5d7.west);
\draw[->] (nb5d7.east) -- (nb5d7b7.west);
\draw[->] (nb5d7.east) -- (nb5d7T.west);
\draw[->] (nb5d7T.east) -- (nb5d7Tb7.west);
\draw[->] (nb5d7T.east) -- (nb5d7Tf5.west);

\draw[->] (nd7.east) -- (nd7b7.west);
\draw[->] (nd7.east) -- (nd7b5.west);
\draw[->] (nd7.east) -- (nd7T.west);
\draw[->] (nd7T.east) -- (nd7Tb7.west);
\draw[->] (nd7T.east) -- (nd7Tf5.west);
\draw[->] (nd7b7.east) -- (nd7b7d7.west);
\draw[->] (nd7b7.east) -- (nd7b7b5.west);
\draw[->] (nd7b7d7.east) -- (nd7b7d7T.west);
\draw[->] (nd7b7d7.east) -- (nd7b7d7b5.west);
\end{tikzpicture}
\]
That is, whenever Alfred jumps to b5, b7, d7, or f5, Betty places a
chap at c5, c6, d6, or f4, respectively; and if the ball is at d7 and
there is a chap at d6 and Alfred tackles at d5, then Betty responds by
re-tackling at d4.

Therefore, Alfred is forced to tackle at d5, but once he does, he has
a shot of his own. Betty can jot off $\Dne\Dw\Ds$ to c4, but Alfred
can simply respond by placing a chap at c5 to get an untackleable
unjottable shot. So Betty is also forced to tackle at e6.

And the tables keep on turning. Now Betty gets a shot and jotting off
to h9 does not help, so Alfred tackles at f7. Then he gets a shot, and
by the previous analysis and the symmetry of the board, we know that
none of Betty's three possible jots to g6, i6, or i8 help. Therefore,
Betty tackles at g8 and gets another shot. Alfred's jot to f7 is no
good, so he tackles at h9, gaining an unjottable shot. Betty places a
chap at i10 which is her only defence.

Alfred now has to be extremely careful about his next move. Betty is
threatening a win in two by first placing a chap at j11, which blocks
all jumps along the $\Dne$ direction, and then placing a chap at a1
for a shot. The only way Alfred can prevent Betty from placing a chap
at j11 is by jumping there himself. If he does not do the jump, then
he has two moves to create a new jump for himself against Betty's
threat.

In one of those two moves, Alfred has to place a chap next to the
ball. Placing a chap at a1 and b1 are always shots for Betty, so he
has to place it at b2 or a3. Placing a chap at b2 is also a shot for
Betty with the jump $\De\Dsw$; if Alfred tries to prevent that by
first placing a chap at d2, and then at b2, it still remains a shot
for Betty via the jump $\De\Dse$.

Placing a chap at a3 is a shot for Betty as well using the jump
$\Dn\Dse$; Alfred can try to defend against that by first placing a
chap at a4, and then at a3. So this is Alfred's only possible defence
against Betty's threat, if he chooses not to jump to j11
immediately.

However, the moment Alfred places a chap at a4, Betty places a chap at
a1, forcing Alfred to jump $\Dne$ to j11. Then we get a board position
which is almost symmetric to the original position, except we have two
additional chaps at a1 and a4. By the previous analysis, Betty and
Alfred are forced to place chaps, one at a time, from i10 to b3. The
extra chap at a4 does not feature in any of relevant sequences, but
the extra chap at a1 comes back to haunt Alfred at the very end, for
when he tries to place a chap at b3, the extra chap at a1 makes it a
shot for Betty. That is, we get the following sequence
\[
\text{$\alpha$(v)a4 $\beta$(v)a1!! $\alpha$(vi)$\Dne$ $\beta$(vi)i10! $\alpha$(vii)h9! $\beta$(vii)g8! $\alpha$(viii)f7! $\beta$(viii)e6! $\alpha$(ix)d5! $\beta$(ix)c4\#}
\]
which loses for Alfred.

Therefore, Alfred is forced to play $\alpha$(v)$\Dne$. It is now
Betty's move, and the position is symmetric to the starting
position. Consequently, this is a draw. (The optimal sequence of moves
is shown in Figure~\ref{fig:draw-sequence}.)
\end{proof}

\begin{corollary}\label{cor:main}
  There are drawn configurations in Phutball played on an $m\times n$
  board with $m-2\geq n\geq 10$.
\end{corollary}

\begin{proof}
  Let us just show how to generalize the configuration from
  Figure~\ref{fig:main-draw} to the configuration of
  Figure~\ref{fig:official-draw} on the official $19\times 15$ board,
  and leave the rest as an exercise to the reader.
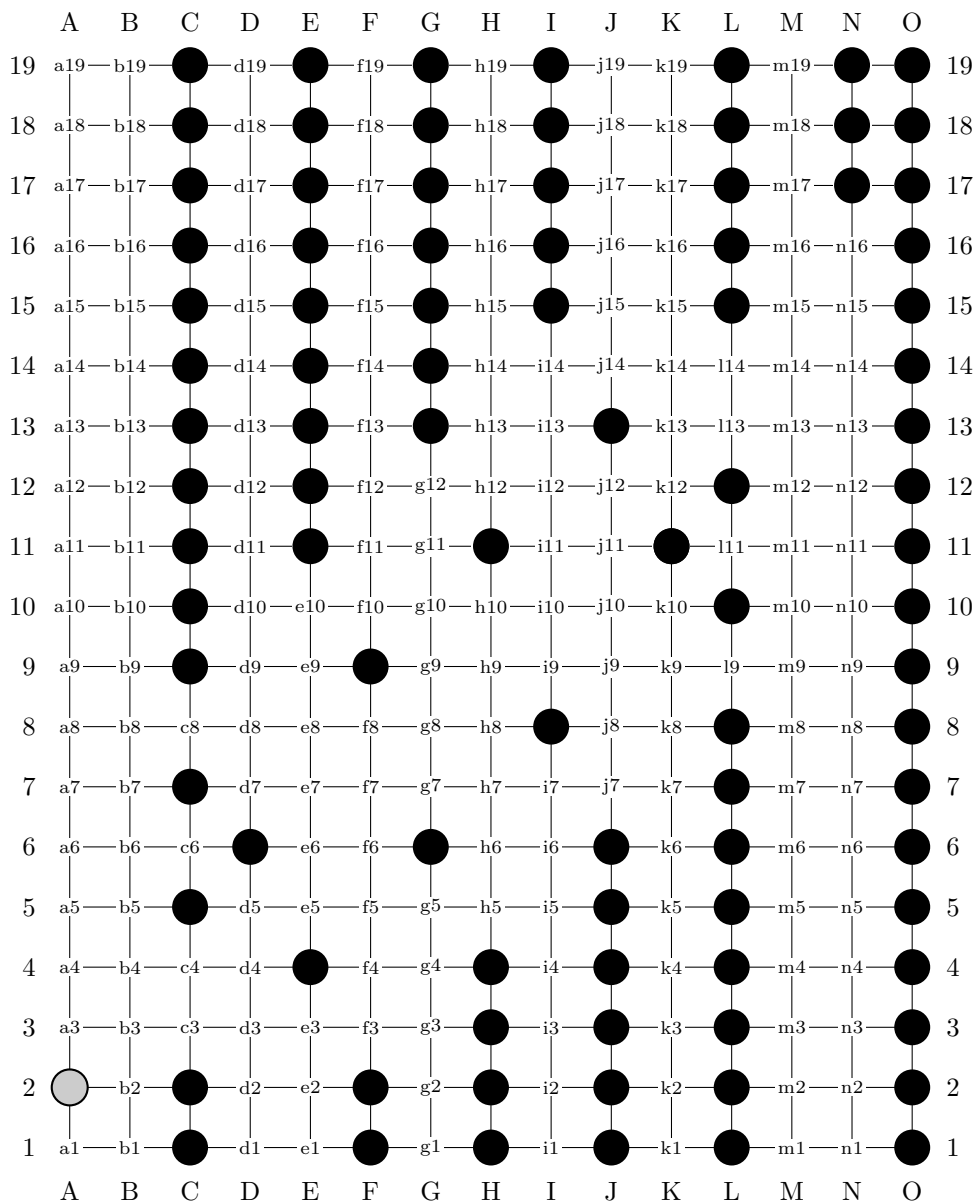
\begin{figure}[ht]
\centering
\begin{tikzpicture}[scale=0.8]

\draw (0,1) grid (14,19);

\foreach \l [count=\i from 0] in {A,B,C,D,E,F,G,H,I,J,K,L,M,N,O}
{
  \node[below=3ex,inner sep=0pt,outer sep=0pt] at (\i,1) {\l};
  \node[above=3ex,inner sep=0pt,outer sep=0pt] at (\i,19) {\l};
}
\foreach \j in {1,...,19}
{
  \node[left=3ex,inner sep=0pt,outer sep=0pt] at (0,\j) {\j};
  \node[right=3ex,inner sep=0pt,outer sep=0pt] at (14,\j) {\j};
}
\foreach \l [count=\i from 0] in {a,b,c,d,e,f,g,h,i,j,k,l,m,n,o}
{
  \foreach \j in {1,...,19}
  {
    \node[fill=white,fill opacity=1, text opacity=1,inner sep=1pt,
    outer sep=0pt] (\l\j) at (\i,\j) {\scriptsize \l\j};
  }
}

\node[ball] at (a2) {};

\foreach \i in {1,...,19}
{
  \node[chap] at (o\i) {};
}

\node[chap] at (c1) {};
\node[chap] at (c2) {};
\node[chap] at (c5) {};
\node[chap] at (c7) {};

\foreach \i in {9,...,19}
{
  \node[chap] at (c\i) {};
}
\node[chap] at (d6) {};

\node[chap] at (e4) {};
\foreach \i in {11,...,19}
{
  \node[chap] at (e\i) {};
}
\foreach \i in {1,2}
{
  \node[chap] at (f\i) {};
}
\node[chap] at (f9) {};
\node[chap] at (g6) {};
\foreach \i in {13,...,19}
{
  \node[chap] at (g\i) {};
}
\foreach \i in {1,...,4}
{
  \node[chap] at (h\i) {};
}
\node[chap] at (h11) {};
\node[chap] at (i8) {};
\foreach \i in {15,...,19}
{
  \node[chap] at (i\i) {};
}
\foreach \i in {1,...,6}
{
  \node[chap] at (j\i) {};
}
\node[chap] at (j13) {};

\node[chap] at (k11) {};
\foreach \i in {1,...,8}
{
  \node[chap] at (l\i) {};
}
\node[chap] at (l10) {};
\node[chap] at (l12) {};
\node[chap] at (k11) {};
\foreach \i in {15,...,19}
{
  \node[chap] at (l\i) {};
}

\foreach \i in {17,...,19}
{
  \node[chap] at (n\i) {};
}
\end{tikzpicture}
\caption{A drawn position in Phutball on the standard pitch.}\label{fig:official-draw}
\end{figure}
The reader should note that despite filling up the rightmost column
with chaps, the sideline subtleties from Figure~\ref{fig:illustrative}
are not invoked. The reader should additionally note that some of the
analysis from the proof of Theorem~\ref{thm:main} become more
involved. For instance, after $\alpha$(i)b3! $\beta$(i)c4!, if Alfred
does not tackle at d5, he still loses, but the following sequence is
longer.
  \[
  \text{$\alpha$(ii)$\Dne\Dn$ $\beta$(ii)d6! $\alpha$(iii)d5! $\beta$(iii)d4!! $\alpha$(iv)$\Ds\Dne$ $\beta$(iv)f4!! $\alpha$(v)$\Dne$ $\beta$(v)h6!! $\alpha$(vi)$\Dne$ $\beta$(iv)j8\#}\qedhere
  \]
\end{proof}

The careful reader will notice that the above configurations do not
generalize to the usual $19\times 19$ board since the optimal sequence
of chap placements in the draw is done along a diagonal,
cf.~Figure~\ref{fig:draw-sequence}, and the diagonal needs go from one
sideline to the other since Alfred is only forced to jump because of
the sidelines, and we need to have a row below and a row above to
ensure that the game is not already over. So we leave that as a
question to the reader.
\begin{question}
  Are there drawn configurations in Phutball played on a $19\times 19$
  board?
\end{question}

\noindent\textbf{Acknowledgment.} The author is grateful to John Conway for
introducing him to Phutball, for teaching him the subtle failure of
the strategy-stealing argument in two-dimensional Phutball, and for
the many great games of Phutball that they played for hours in John's
office (the Fine Hall common room). He would also like to thank the
referees for their comments and suggestions. And thank you Simon!

\bibliographystyle{amsalpha}
\bibliography{phutball}

\providecommand{\bysame}{\leavevmode\hbox to3em{\hrulefill}\thinspace}
\providecommand{\MR}{\relax\ifhmode\unskip\space\fi MR }
\providecommand{\MRhref}[2]{%
  \href{http://www.ams.org/mathscinet-getitem?mr=#1}{#2}
}
\providecommand{\href}[2]{#2}
\begin{thebibliography}{BCG03}

\bibitem[BCG03]{BCG}
Elwyn~R. Berlekamp, John~H. Conway, and Richard~K. Guy, \emph{Winning ways for
  your mathematical plays. {V}ol. 3}, second ed., A K Peters, Ltd., Natick, MA,
  2003. \MR{2006327}

\bibitem[DDE02]{phutball-np}
Erik~D. Demaine, Martin~L. Demaine, and David Eppstein, \emph{Phutball endgames
  are hard}, More games of no chance, Math. Sci. Res. Inst. Publ., vol.~42,
  Cambridge Univ. Press, Cambridge, 2002, pp.~351--360. \MR{1973023}

\bibitem[Der10]{phutball-pspace}
Dariusz Dereniowski, \emph{Phutball is {PSPACE}-hard}, Theoret. Comput. Sci.
  \textbf{411} (2010), no.~44-46, 3971--3978. \MR{2768776}

\bibitem[GN02]{phutball-1d}
J.~P. Grossman and Richard~J. Nowakowski, \emph{One-dimensional {P}hutball},
  More games of no chance, Math. Sci. Res. Inst. Publ., vol.~42, Cambridge
  Univ. Press, Cambridge, 2002, pp.~361--367. \MR{1973024}

\bibitem[Loo08]{phutball-directional}
Anne~M. Loosen, \emph{Directional phutball}, Atl. Electron. J. Math. \textbf{3}
  (2008), no.~1, 30--45. \MR{2827227}

\bibitem[Sie09]{phutball-loopy}
Aaron~N. Sieger, \emph{Coping with cycles}, Games of No Chance 3, Math. Sci.
  Res. Inst. Publ., vol.~56, Cambridge University Press, New York, NY, USA, 1st
  ed., 2009, pp.~215--232.

\end{thebibliography}

\end{document}